\documentclass[sn-mathphys,Numbered]{sn-jnl}

\usepackage{graphicx}%
\usepackage{multirow}%
\usepackage{amsmath,amssymb,amsfonts}%
\usepackage{amsthm}%
\usepackage{mathrsfs}%
\usepackage[title]{appendix}%
\usepackage{xcolor}%
\usepackage{textcomp}%
\usepackage{manyfoot}%
\usepackage{booktabs}%
\usepackage{algorithm}%
\usepackage{algorithmicx}%
\usepackage{algpseudocode}%
\usepackage{listings}%

\theoremstyle{thmstyleone}%
\newtheorem{theorem}{Theorem}
\newtheorem{proposition}[theorem]{Proposition}%
\theoremstyle{thmstyletwo}%

\theoremstyle{thmstylethree}%
\newtheorem{definition}{Definition}%

\begin{document}

\title[O.L.W.M]{ON LOGIT WEIBULL MANIFOLD.}


\author*[1]{\fnm{} \sur{Prosper Rosaire Mama  Assandje  }}\email{mamarosaire@facsciences-uy1.cm}
\author[2]{\fnm{} \sur{Joseph Dongho }}\email{josephdongho@yahoo.fr}
\equalcont{These authors contributed equally to this work.}
\author[3]{\fnm{} \sur{ Thomas Bouetou Bouetou}}\email{tbouetou.@gmail.com}
\equalcont{These authors contributed equally to this work.}

\affil*[1]{\orgdiv{Department of Mathematics, Faculty of sciences},
\orgname{University of Yaounde I},
\orgaddress{\street{usrectorat@.univ-yaounde1.cm}, \city{Yaounde},
\postcode{337}, \state{Center}, \country{Cameroon}}}

\affil[2]{\orgdiv{Department of Mathematics and Computer Science},
\orgname{University of Maroua},
\orgaddress{\street{decanat@fs.univ-maroua.cm}, \city{Maroua},
\postcode{814}, \state{Far -North}, \country{Cameroon}}}

\affil[3]{\orgdiv{Computer Engineering Department.}, \orgname{Higher
national school of Polytechnic},
\orgaddress{\street{thomas.bouetou@polytechnique.cm}, \city{Yaounde
I}, \postcode{8390}, \state{Center}, \country{Cameroon}}}



\abstract{In this work, it is shown that there is no potential
function on the Weilbull statistical manifold. However, from the
two-parameter Weibull model we can extract a model with a potential
function called the logit model. On this logit model, there is a
completely integrable Hamiltonian gradient system.}




\maketitle

\section{Introduction}\label{sec1}

The idea of this paper is to show that from a Weibull manifold with
no potential function one can extract a hybrid manifold possessing
the properties of geometric invariants, one of which is the
existence of the potential function on the manifold, in order to
show that one can construct a gradient system on this manifold and
show that it is Hamiltonian and completely integrable. How can we
construct a gradient system on a non-potential Weibull manifold and
show that it is Hamiltonian and fully integrable? On this question,
we have the following research questions: How to extract the How to
construct a gradient system in such a manifold? Work in this area
goes back to Amari's\cite{Shu,Shu1}, which gives existence
properties of the potential function on a statistical manifold and
show that under certain conditions of geometric invariance the
Riemannian metric and the Fisher information metric. Fujiwara
\cite{Dy,Dy1} and
Nakamura\cite{nakamura,nakamura1,nakamura2,nakamura4}, show that on
an even-dimensional statistical manifold admitting potential
functions, there exists a completely integrable Hamiltonian gradient
system. In \cite{mama}, we show that the gradient system defined on
a lognormal manifold is a Hamiltonian and completely integrable
system on this manifold. Hisatoshi-Tanaka \cite{hisatoshi} consider
parametric binary choice models from the perspective of information
geometry. The set of models is a dually flat manifold with dual
connections, which are naturally derived from the Fisher information
metric. Under the dual connections, the canonical divergence and the
Kullback-Leibler(KL) divergence of the binary choice model coincide
if and only if the model is a logit \cite{hisatoshi}. The results
are applied to a logit estimation with linear constraints. It
proposes logit models allowing the extraction of the potential
function on the manifold, based on the choice of the conditional
probability. In the same, we show that  on the  Weibull statistical
manifold. $\mathrm{I\!
E}\left[\partial_{\theta_{i}}\ell(x,\theta)\right]=0, \; \textrm{for
all}\; i\in \{1,2\}$ if only if\begin{enumerate}
    \item[1)] $\mathrm{I\!
E}\left[x^{b}\right]=a^{b}$
    \item[2)] $\mathrm{I\!
E}\left[\log(x)\right]=-a+(1-\kappa)b+ab\log(a)$,\;and\\
$\mathrm{I\!
E}\left[x^{b}\log(x)\right]=a^{b}\left(\frac{1}{b}-a+(1-\kappa)b+ab\log(a)\right)$
\end{enumerate} where  $\kappa$ be Euler's constant. We show that
The Riemannian metric on the Weibull
 manifold is given by
 \begin{eqnarray*}G&=&\left(
                                             \begin{array}{cc}
                                               \frac{b^{2}}{a^{2}} &\frac{\varrho_{1}-1}{a} \\
                                               \frac{\varrho_{1}-1}{a} &\frac{b\pi^{2}-6a^{2}\varrho_{2}}{6a^{2}}\\
                                             \end{array}
                                           \right)
 \end{eqnarray*}
where $\varrho_{1}=ab+b(1-ab)\log(a)-(1-\kappa)b^{2}$, and $
\varrho_{2}= -\frac{1}{b^{2}}-2\left(\frac{1}{b}-a+
(1-\kappa)b+2\vartheta\right)\log(a)-(1++2a+b)\log^{2}(a)-b\vartheta^{2}$
with
 $\vartheta=-\frac{1}{b}+\frac{1-\kappa}{a}+\log(a)$, and $\kappa$ be Euler's constant.
 The inverse matrix is given by
 \begin{eqnarray*}
 G^{-1}&=&\left(
    \begin{array}{cc}
     \frac{a^{2}\left(b\pi^{2}-6a^{2}\varrho_{2}\right)}{b^{3}\pi^{2}-6a^{2}b^{2}\varrho_{2}-6a^{2}+12a^{2}\varrho_{1}-6a^{2}\varrho_{1}^{2}} &-\frac{6a^{3}\left(-1+\varrho_{1}\right)}{b^{3}\pi^{2}-6a^{2}b^{2}\varrho_{2}-6a^{2}+12a^{2}\varrho_{1}-6a^{2}\varrho_{1}^{2}} \\
                                               -\frac{6a^{3}\left(-1+\varrho_{1}\right)}{b^{3}\pi^{2}-6a^{2}b^{2}\varrho_{2}-6a^{2}+12a^{2}\varrho_{1}-6a^{2}\varrho_{1}^{2}} & \frac{6a^{2}b^{2}}{b^{3}\pi^{2}-6a^{2}b^{2}\varrho_{2}-6a^{2}+12a^{2}\varrho_{1}-6a^{2}\varrho_{1}^{2}} \\
                                             \end{array}
                                           \right)
 \end{eqnarray*}
 This leads us to show that, on the Weibull statistical  model where $p_{ \theta}$ is
Weibull density function. The coordinate system on  Weibull manifold
does not admit dual coordinates or potential function. So, Having
defined the product on $\mathrm{I\! R}^{2}$, we show that there is
an action $\nu$ on $\mathrm{I\! R}^{2}$, that satisfies the
regularity conditions given by \begin{eqnarray*}
\nu:\mathrm{I\! R}^{2}\times\mathrm{I\! R}&\longrightarrow&\mathrm{I\! R} \\
  (\theta,x)&=&x\cdot\theta =a^{-b}x^{b}
\end{eqnarray*}
we show that, for all  Weibull statistical manifold $S = \left\{p_{
\theta}(x)=\frac{b}{a}\left(\frac{x}{a}\right)^{b-1}e^{-\left(\frac{x}{a}\right)^{b}},\left.
                        \begin{array}{ll}
                         \theta= (a , b )\in \mathrm{I\! R}_{+}\times \mathrm{I\! R}_{+} & \hbox{} \\
                          x\in \mathrm{I\! R}_{+} & \hbox{}
                        \end{array}
                      \right.
\right\}$   where $p_{ \theta}$ is Weibull density function, there
exist the logit model\\ $S' = \left\{p_{
\theta}(y,x)=\frac{b}{2a}\left(\frac{x}{a}\right)^{b-1}e^{-\left(\frac{x}{a}\right)^{b}},\left.
                        \begin{array}{ll}
                         \theta= (a , b )\in \mathrm{I\! R}_{+}\times \mathrm{I\! R}_{+} & \hbox{} \\
                          x\in \mathrm{I\! R}_{+} & \hbox{}
                        \end{array}
                      \right.
\right\}$ with the fundamental condition on the variable
\[\displaystyle {\it x}={\it RootOf} \left( 2\,b\,{a}^{-b}{{\it \_Z}}^{b}-2\,b-2\,{a}^{-b}{{\it \_Z}}^{b}a\,\ln  \left( a \right) \\
\mbox{}+2\,{a}^{-b}{{\it \_Z}}^{b}a\,\ln  \left( {\it \_Z} \right)
-2\,\ln  \left( {\it \_Z} \right) a+2\,\ln \left( a \right) a-a
\right)\] admitting the potential function \begin{eqnarray*}
  \Phi(\theta) &=&
  \frac{b^{2}}{12a^{2}x}\left(a^{-b}x^{b}-1\right)^{4}+\frac{b^{2}}{3a^{2}x}a^{-b}x^{b}-\frac{b^{2}}{12a^{2}x}.
\end{eqnarray*}, the dual coordinate system given by
    $\left(\eta_{1},\eta_{2}\right)=\left(\xi_{1}(\theta),\xi_{2}(\theta)\right),$
and dual potential function
\begin{equation*}
  \Psi(\eta)= a\eta_{1}+b\eta_{2}-  \frac{b^{2}}{12a^{2}x}\left(a^{-b}x^{b}-1\right)^{4}+\frac{b^{2}}{3a^{2}x}a^{-b}x^{b}-\frac{b^{2}}{12a^{2}x}
\end{equation*}
 satisfy the Legendre equation
\begin{equation*}
   \theta_{1}\eta_{1}+\theta_{2}\eta_{2}- \Phi(\theta)-\Psi(\eta)=0
\end{equation*}
In the same, we show that on logit  Weibull manifold $S' =
\left\{p_{
\theta}(y,x)=\frac{b}{2a}\left(\frac{x}{a}\right)^{b-1}e^{-\left(\frac{x}{a}\right)^{b}},\left.
                        \begin{array}{ll}
                         \theta= (a , b )\in \mathrm{I\! R}_{+}\times \mathrm{I\! R}_{+} & \hbox{} \\
                          x\in \mathrm{I\! R}_{+} & \hbox{}
                        \end{array}
                      \right.
\right\}$, the gradient system on logit Weibull manifold is given by
\begin{equation*}
    \left\{
  \begin{array}{ll}
     \dot{a}=\frac{1}{\mathcal{A}}\frac{\partial^{2}\psi(\theta)}{\partial b^{2}} .\frac{\partial\psi(\theta)}{\partial a}
     - \frac{1}{\mathcal{A}}\frac{\partial^{2}\psi(\theta)}{\partial a\partial b}.\frac{\partial\psi(\theta)}{\partial b}& \hbox{} \\
    \dot{b}= -\frac{1}{\mathcal{A}}\frac{\partial^{2}\psi(\theta)}{\partial a\partial b}.\frac{\partial\psi(\theta)}{\partial a}
     +\frac{1}{\mathcal{A}}\frac{\partial^{2}\psi(\theta)}{\partial a^{2}}.\frac{\partial\psi(\theta)}{\partial b}& \hbox{}
  \end{array}
\right.
\end{equation*}
where \begin{eqnarray*}
  \Phi(\theta) &=&
  \frac{b^{2}}{12a^{2}x}\left(a^{-b}x^{b}-1\right)^{4}+\frac{b^{2}}{3a^{2}x}a^{-b}x^{b}-\frac{b^{2}}{12a^{2}x}.
\end{eqnarray*} and \[\displaystyle {\it x}={\it RootOf} \left( 2\,b\,{a}^{-b}{{\it \_Z}}^{b}-2\,b-2\,{a}^{-b}{{\it \_Z}}^{b}a\,\ln  \left( a \right) \\
\mbox{}+2\,{a}^{-b}{{\it \_Z}}^{b}a\,\ln  \left( {\it \_Z} \right)
-2\,\ln  \left( {\it \_Z} \right) a+2\,\ln \left( a \right) a-a
\right).\]
 After the introduction,
 the first section \ref{sec2} recall the preliminaries motion on theory of statistical manifold, in section \ref{sec3}  we
determine the Riemannian Riemannian metric on Weibull statistical
manifold, in section \ref{sec4}, we determine geometry properties on
Weibull distribution.in section \ref{sec5}, we determine the
potential function and gradient system on Weibull logit manifold.

\section{Preliminaries}\label{sec2}
Let $S = \left\{p_{ \theta}(x),\left.
                        \begin{array}{ll}
                         \theta\in \Theta & \hbox{} \\
                          x\in \mathcal{X} & \hbox{}
                        \end{array}
                      \right.
\right\}$ be the set of probabilities $p_{ \theta}$, parameterized
by $ \Theta$, open a subset of $\mathrm{I\! R}^{n}$; on the sample
space $\mathcal{X}\subseteq\mathrm{I\! R}$.  Let
$\mathcal{F(\mathcal{X},\mathrm{I\! R})}$ be the space of
real-valued smooth functions on $\mathcal{X}$. According to Ovidiu
\cite{ ovidiu-book3}, the log-likelihood function is a mapping
defined by
\begin{eqnarray}
l:S&\longrightarrow& \mathcal{F(\mathcal{X},\mathrm{I\! R})}\nonumber\\
p_{ \theta} &\longmapsto&  l\left(p_{ \theta}\right)(x) = \log
p_{\theta}(x)\nonumber
\end{eqnarray}
 Sometimes, for convenient reasons, this will be denoted by
$l(x,\theta)=l\left(p_{ \theta}\right)(x)$.\\
In \cite{nak-journal} and \cite{Shu-book1}, the Fisher information
defined by
\begin{equation}(g_{ij})_{1\leq i;j\leq n}=\left(-\mathrm{I\! E}[\partial_{\theta_{i}}\partial_{\theta_{j}}l(x,\theta)]\right)_{1\leq i;j\leq n}\label{e0}
\end{equation}
According Amari's \cite{Shu},two basis vectors  are said to be
biorthogonal $\partial_{\theta_{i}}$ and $\partial_{\eta_{j}}$ if it
satisfy
\begin{equation}
    \langle\partial_{\theta_{i}},\partial_{\eta_{j}}\rangle=\delta_{i}^{j}, \textrm{with}\; \partial_{\theta_{i}}:=\frac{\partial}{\partial\theta_{i}}.
\end{equation}

According Amari's theorem \cite{Shu}, When a Riemannian manifold $S$
has a pair of dual coordinate systems $(\theta, \eta)$, there exist
potential functions $\Phi$ and $\phi$ such that the metric tensors
are derived by
\begin{equation*}
    g_{ij}=\partial_{\theta_{i}}\partial_{\theta_{j}}\Phi(\theta),\;g^{ij}=\partial_{\eta_{i}}\partial_{\eta_{j}}\Psi(\eta),\textrm{with}\; \partial_{\theta_{i}}:=\frac{\partial}{\partial\theta_{i}}.
\end{equation*}
 Conversely, when either potential function $\Phi$ or $\Psi$ exists from
which the metric is derived by differentiating it twice, there exist
a pair of dual coordinate systems. The dual coordinate systems are
related by the following Legendre transformations
\begin{equation}
\theta_{i}=\partial_{\eta_{i}}\Psi(\eta),\;\eta_{i}=\partial_{\theta_{i}}\Phi(\theta)
\end{equation} where the two
potential functions satisfy the identity
\begin{equation}
\Phi+\Psi-\theta_{i}\eta_{i}=0.
\end{equation}
 Denote $G=(g_{ij})_{1\leq i;j\leq
n}$ the Fisher information matrix, the gradient system is given by
\begin{equation}\dot{\overrightarrow{\theta}}=-G^{-1}\partial_{\theta}\Phi(\theta).\label{e2}\end{equation}
The complete integrability of gradient system (\ref{e2}) is proven
if the Theorem 1 in \cite{mama-proceeding} is verify.

\section{Riemannian metric on Weibull statistical manifold}\label{sec3}

\begin{proposition}\label{pro1}
Let $S = \left\{p_{
\theta}(x)=\frac{b}{a}\left(\frac{x}{a}\right)^{b-1}e^{-\left(\frac{x}{a}\right)^{b}},\left.
                        \begin{array}{ll}
                         \theta= (a , b )\in \mathrm{I\! R}_{+}\times \mathrm{I\! R}_{+} & \hbox{} \\
                          x\in \mathrm{I\! R}_{+} & \hbox{}
                        \end{array}
                      \right.
\right\}$ be a  Weibull statistical  model where $p_{ \theta}$ is
Weibull density function. Let
$\mathcal{B}^{\ell}=\left\{\partial_{a}\ell(x,\theta)=\frac{b}{a}\left(a^{-b}x^{b}-1\right);
\partial_{b}\ell(x,\theta)=a^{-b}x^{b}\log(a)-a^{-b}x^{b}\log(x)+\log(x)-\log(a)+\frac{1}{b}\right\}$
the natural basis of the tangent space in one point $p$ of the
Weibull statistical manifold. $\mathrm{I\!
E}\left[\partial_{\theta_{i}}\ell(x,\theta)\right]=0, \; \textrm{for
all}\; i\in \{1,2\}$ if only if\begin{enumerate}
    \item[1)] $\mathrm{I\!
E}\left[x^{b}\right]=a^{b}$
    \item[2)] $\mathrm{I\!
E}\left[\log(x)\right]=-a+(1-\kappa)b+ab\log(a)$,\;and\\
$\mathrm{I\!
E}\left[x^{b}\log(x)\right]=a^{b}\left(\frac{1}{b}-a+(1-\kappa)b+ab\log(a)\right)$
\end{enumerate} where  $\kappa$ be Euler's constant.
\end{proposition}
\begin{proof}
Let $p_{
\theta}(x)=\frac{b}{a}\left(\frac{x}{a}\right)^{b-1}e^{-\left(\frac{x}{a}\right)^{b}}$
be a Weibull density function. We have $\ell(x,\theta)=\log p_{
\theta}(x)$. So, we have $\ell(x,\theta)=\log (b)-\log(a)-(b-1)\log
(a)+(b-1)\log(x)-a^{-b}x^{b}$. We obtain
$\partial_{a}\ell(x,\theta)=\frac{b}{a}\left(a^{-b}x^{b}-1\right)$
and
$\partial_{b}\ell(x,\theta)=a^{-b}x^{b}\log(a)-a^{-b}x^{b}\log(x)+\log(x)-\log(a)+\frac{1}{b}$.
Therefore
\begin{enumerate}
    \item [(1)]If $\mathrm{I\!
E}\left[\partial_{a}\ell(x,\theta)\right]=0$ then we have
$\mathrm{I\! E}\left[x^{b}\right]=a^{b}$.
    \item [(2)] $\mathrm{I\!
E}\left[\partial_{b}\ell(x,\theta)\right]=0$ then we have
$\mathrm{I\!E}\left[a^{-b}x^{b}\log(a)-a^{-b}x^{b}\log(x)+\log(x)-\log(a)+\frac{1}{b}\right]=0$.
We obtain \begin{eqnarray}\label{ln1}\mathrm{I\!
E}\left[x^{b}\log(x)\right]&=&\frac{a^{b}}{b}+a^{b}\mathrm{I\!
E}\left[\log(x)\right]\end{eqnarray}. Or
 \begin{eqnarray*}\mathrm{I\!
E}\left[\log(x)\right]&=&\int_{0}^{+\infty}\log(x)\frac{b}{a}\left(\frac{x}{a}\right)^{b-1}e^{-\left(\frac{x}{a}\right)^{b}}dx\end{eqnarray*}
by setting $\xi=\log(x), \; x d\xi=dx,\; and\; x=e^{\xi}.$ We have
\begin{eqnarray*}\mathrm{I\!
E}\left[\log(x)\right]&=&\int_{-\infty}^{+\infty}\xi\frac{b}{a}\left(\frac{e^{\xi}}{a}\right)^{b-1}e^{-\left(\frac{e^{\xi}}{a}\right)^{b}}e^{\xi}d\xi\end{eqnarray*}
So, we have
\begin{eqnarray*}\mathrm{I\!
E}\left[\log(x)\right]&=& \int_{-\infty}^{+\infty}\xi
 b\left(\frac{e^{\xi}}{a}\right)^{b}e^{-\left(\frac{e^{\xi}}{a}\right)^{b}}d\xi.\end{eqnarray*}
Let $q_{
\theta}(\zeta)=\frac{1}{a}e^{-e^{-\frac{\zeta-b}{a}}}.e^{-\frac{\zeta-b}{a}},\;
-\infty<\zeta<+\infty,\;(a , b )\in \mathrm{I\! R}_{+}\times
\mathrm{I\! R}_{+}$ the Gumbel distribution, with
\begin{eqnarray}\label{euler1}\mathrm{I\! E}\left[\zeta\right]&=&a+b\kappa,\; V(\zeta)=\frac{\pi^{2}b^{2}}{6};\end{eqnarray} where
$\kappa$ is Euler constant. By setting $\gamma=-\frac{\zeta-b}{a}$,
we have $q_{ \theta}(\gamma)=\frac{1}{a}e^{-e^{\gamma}}.e^{\gamma}$
and $V(\gamma)=\frac{1}{a^{2}}V(\zeta)$.  The relation
(\ref{euler1}) becomes
\begin{eqnarray}\label{euler2}\mathrm{I\!
E}\left[\gamma\right]&=&-1+(1-\kappa)\frac{b}{a}.\end{eqnarray} The
same, by setting $e^{\gamma}=\left(\frac{e^{\xi}}{a}\right)^{b},\;
\gamma=b\xi-b\log(a)$; we have $V(\xi)=\frac{1}{b^{2}}V(\gamma)$
\begin{eqnarray}\label{de1}q_{\theta}(\xi)&=&
\frac{1}{a}\left(\frac{e^{\xi}}{a}\right)^{b}
e^{-\left(\frac{e^{\xi}}{a}\right)^{b}},\;
V(\xi)=\frac{\pi^{2}}{6a^{2}}\end{eqnarray} and (\ref{euler2})
becomes
\begin{eqnarray}\label{euler3}\mathrm{I\!
E}\left[\xi\right]&=&-\frac{1}{b}+\frac{(1-\kappa)}{a}+\log(a).\end{eqnarray}
We write \begin{eqnarray}\label{ln2}\mathrm{I\!
E}\left[\log(x)\right]&=&ba
\int_{-\infty}^{+\infty}\xi\frac{1}{a}\left(\frac{e^{\xi}}{a}\right)^{b}e^{-\left(\frac{e^{\xi}}{a}\right)^{b}}d\xi.\end{eqnarray}i.e.
\begin{eqnarray}\label{ln0}\mathrm{I\!
E}\left[\log(x)\right]&=&ba\mathrm{I\!
E}\left[\xi\right].\end{eqnarray} Using (\ref{euler3}) in
(\ref{ln0}) we obtain,
\begin{eqnarray}\label{ln3}\mathrm{I\!
E}\left[\log(x)\right]&=&-a+(1-\kappa)b+ab\log(a).\end{eqnarray}
Using (\ref{ln3}) in (\ref{ln1}) we have
\begin{eqnarray}\label{ln4}\mathrm{I\!
E}\left[x^{b}\log(x)\right]&=&\frac{a^{b}}{b}-a^{b+1}+(1-\kappa)a^{b}b+a^{b+1}b\log(a).\end{eqnarray}
\end{enumerate}
\end{proof}

In the following we put ourselves in the conditions of
proposition\ref{pro1}. We have the following proposition
\begin{proposition}\label{pro2}
Let $S = \left\{p_{
\theta}(x)=\frac{b}{a}\left(\frac{x}{a}\right)^{b-1}e^{-\left(\frac{x}{a}\right)^{b}},\left.
                        \begin{array}{ll}
                         \theta= (a , b )\in \mathrm{I\! R}_{+}\times \mathrm{I\! R}_{+} & \hbox{} \\
                          x\in \mathrm{I\! R}_{+} & \hbox{}
                        \end{array}
                      \right.
\right\}$ be a  Weibull statistical  model where $p_{ \theta}$ is
Weibull density function. The Riemannian metric on the Weibull
 manifold is given by
 \begin{eqnarray*}G&=&\left(
                                             \begin{array}{cc}
                                               \frac{b^{2}}{a^{2}} &\frac{\varrho_{1}-1}{a} \\
                                               \frac{\varrho_{1}-1}{a} &\frac{b\pi^{2}-6a^{2}\varrho_{2}}{6a^{2}}\\
                                             \end{array}
                                           \right)
 \end{eqnarray*}
where $\varrho_{1}=ab+b(1-ab)\log(a)-(1-\kappa)b^{2}$, and $
\varrho_{2}= -\frac{1}{b^{2}}-2\left(\frac{1}{b}-a+
(1-\kappa)b+2\vartheta\right)\log(a)-(1++2a+b)\log^{2}(a)-b\vartheta^{2}$
with
 $\vartheta=-\frac{1}{b}+\frac{1-\kappa}{a}+\log(a)$, and $\kappa$ be Euler's constant.
 The inverse matrix is given by
 \begin{eqnarray*}
 G^{-1}&=&\left(
    \begin{array}{cc}
     \frac{a^{2}\left(b\pi^{2}-6a^{2}\varrho_{2}\right)}{b^{3}\pi^{2}-6a^{2}b^{2}\varrho_{2}-6a^{2}+12a^{2}\varrho_{1}-6a^{2}\varrho_{1}^{2}} &-\frac{6a^{3}\left(-1+\varrho_{1}\right)}{b^{3}\pi^{2}-6a^{2}b^{2}\varrho_{2}-6a^{2}+12a^{2}\varrho_{1}-6a^{2}\varrho_{1}^{2}} \\
                                               -\frac{6a^{3}\left(-1+\varrho_{1}\right)}{b^{3}\pi^{2}-6a^{2}b^{2}\varrho_{2}-6a^{2}+12a^{2}\varrho_{1}-6a^{2}\varrho_{1}^{2}} & \frac{6a^{2}b^{2}}{b^{3}\pi^{2}-6a^{2}b^{2}\varrho_{2}-6a^{2}+12a^{2}\varrho_{1}-6a^{2}\varrho_{1}^{2}} \\
                                             \end{array}
                                           \right)
 \end{eqnarray*}
\end{proposition}
\begin{proof}
Let
$\partial_{a}\ell(x,\theta)=\frac{b}{a}\left(a^{-b}x^{b}-1\right)$
and
$\partial_{b}\ell(x,\theta)=a^{-b}x^{b}\log(a)-a^{-b}x^{b}\log(x)+\log(x)-\log(a)+\frac{1}{b}$.
we have
\begin{eqnarray*}\partial_{a}\partial_{a}\ell(x,\theta)&=&-\frac{b}{a^{2}}
\left(-1+a^{-b}bx^{b}+a^{-b}x^{b}\right)\\
\partial_{a}\partial_{b}\ell(x,\theta)&=&
-\frac{1}{a}
\left(1+a^{-b}\log(a)bx^{b}-a^{-b}x^{b}-a^{-b}bx^{b}\log(x)\right)\\
\partial_{b}\partial_{a}\ell(x,\theta)&=&
-\frac{1}{a}
\left(1+a^{-b}\log(a)bx^{b}-a^{-b}x^{b}-a^{-b}bx^{b}\log(x)\right)\\
\partial_{b}\partial_{b}\ell(x,\theta)&=&
-\frac{1}{b^{2}}
\left(1+a^{-b}\log^{2}(a)b^{2}x^{b}-2a^{-b}\log(a)b^{2}x^{b}
\log(x)+a^{-b}b^{2}x^{b}\log^{2}(x)\right)
\end{eqnarray*}
Therefore we have
\begin{eqnarray}\label{ln00}\mathrm{I\!
E}\left[x^{b}\log^{2}(x)\right]&=&
\int_{0}^{+\infty}x^{b}\log^{2}(x)\frac{b}{a}
\left(\frac{x}{a}\right)^{b-1}e^{-\left(\frac{x}{a}\right)^{b}}
dx.\end{eqnarray} By setting $\log(x)=t$, we obtain
\begin{eqnarray}\label{ln0001}\mathrm{I\!
E}\left[x^{b}\log^{2}(x)\right]&=&
\int_{-\infty}^{+\infty}\left(e^{t}\right)^{2}t^{2}\frac{b}{a}
\left(\frac{e^{t}}{a}\right)^{b-1}e^{-\left(\frac{e^{t}}{a}\right)^{b}}e^{t}dt.\end{eqnarray}
By setting $e^{Y}=\frac{e^{t}}{a}$, (\ref{ln0001}) becomes
\begin{eqnarray}\label{ln0002}\mathrm{I\!
E}\left[x^{b}\log^{2}(x)\right]&=& a^{b}b
\int_{-\infty}^{+\infty}Y^{2}\frac{1}{a}
\left(\frac{e^{Y}}{a}\right)^{b}e^{-\left(\frac{e^{Y
}}{a}\right)^{b}}dY\\
&+& 2a^{b}\log(a) \int_{-\infty}^{+\infty}Y\frac{1}{a}
\left(\frac{e^{Y}}{a}\right)^{b}e^{-\left(\frac{e^{Y
}}{a}\right)^{b}}dY\\
&+& a^{b}b\log^{2}(a)\int_{-\infty}^{+\infty}\frac{1}{a}
\left(\frac{e^{Y}}{a}\right)^{b}e^{-\left(\frac{e^{Y
}}{a}\right)^{b}}dY.\end{eqnarray} Using (\ref{de1}) we have

\begin{eqnarray}\label{ln0003}\mathrm{I\!
E}\left[x^{b}\log^{2}(x)\right]&=& a^{b}b \mathrm{I\!
E}\left[Y^{2}\right]+2a^{b}\log(a) \mathrm{I\! E}\left[Y\right]+
a^{b}b\log^{2}(a).\end{eqnarray} Where

\begin{eqnarray*}\mathrm{I\!
E}\left[Y\right]&=& \int_{-\infty}^{+\infty}Y\frac{1}{a}
\left(\frac{e^{Y}}{a}\right)^{b}e^{-\left(\frac{e^{Y
}}{a}\right)^{b}}dY\\
&=&-\frac{1}{b}+\frac{1-\kappa}{a}+\log(a)\end{eqnarray*} We have

\begin{eqnarray*}V(Y)&=&=
\frac{\pi^{2}}{6a^{2}}.\end{eqnarray*}

So we have
\begin{eqnarray*} \mathrm{I\!
E}\left[Y^{2}\right]&=&V(Y)+\mathrm{I\!
E}^{2}\left[Y\right].\end{eqnarray*}
 The relation
\label{ln0003} becomes
\begin{eqnarray}\label{ln0004}\mathrm{I\!
E}\left[x^{b}\log^{2}(x)\right]&=&
\frac{a^{b}b}{6a^{2}}\pi^{2}+a^{b}b\left[-\frac{1}{b}+\frac{(1-\kappa)}{a}+\log(a)\right]^{2}\\
&+&
2a^{b}\log(a)\left[-\frac{1}{b}+\frac{(1-\kappa)}{a}+\log(a)\right]+
a^{b}b\log^{2}(a).
\end{eqnarray}

we have
\begin{eqnarray*}\mathrm{I\!
E}\left[\partial_{a}\partial_{a}\ell(x,\theta)\right]&=&-\frac{b^{2}}{a^{2}}\\
\mathrm{I\!
E}\left[\partial_{a}\partial_{b}\ell(x,\theta)\right]&=&- \frac{b}{a}(1-ab)\log(a) +(1-\kappa)\frac{b^{2}}{a}+\frac{1-ab}{a}\\
\mathrm{I\!
E}\left[\partial_{b}\partial_{a}\ell(x,\theta)\right]&=& - \frac{b}{a}(1-ab)\log(a) +(1-\kappa)\frac{b^{2}}{a}+\frac{1-ab}{a}\\
\mathrm{I\! E}\left[\partial_{b}\partial_{b}\ell(x,\theta)\right]&=&
-\frac{1}{b^{2}}-\frac{b\pi^{2}}{6a^{2}}-2\left(\frac{1}{b}-a+
(1-\kappa)b+2\vartheta\right)\log(a)-(1++2a+b)\log^{2}(a)-b\vartheta^{2}
\end{eqnarray*}
with $\vartheta=-\frac{1}{b}+\frac{1-\kappa}{a}+\log(a)$. By setting
\begin{eqnarray*}\varrho_{1}&=&ab+b(1-ab)\log(a)-(1-\kappa)b^{2}\\
\varrho_{2}&=& -\frac{1}{b^{2}}-2\left(\frac{1}{b}-a+
(1-\kappa)b+2\vartheta\right)\log(a)-(1++2a+b)\log^{2}(a)-b\vartheta^{2}
\end{eqnarray*}
we have the following coefficient
\begin{eqnarray*}
    g_{11}(\theta)&=&\frac{b^{2}}{a^{2}};\\
    g_{12}(\theta)&=&g_{21}(\theta)=\frac{1-\varrho_{1}}{a};\\
    g_{22}(\theta)&=&\frac{b\pi^{2}-6a^{2}\varrho_{2}}{6a^{2}}.
\end{eqnarray*}
We have the following matrix
 \begin{eqnarray*}G&=&\left(
                                             \begin{array}{cc}
                                               \frac{b^{2}}{a^{2}} &\frac{\varrho_{1}-1}{a} \\
                                               \frac{\varrho_{1}-1}{a} &\frac{b\pi^{2}-6a^{2}\varrho_{2}}{6a^{2}}\\
                                             \end{array}
                                           \right)
 \end{eqnarray*}
\end{proof}

\section{Geometry properties  on Weibull distribution}\label{sec4}

\begin{proposition}\label{th1}
Let $S = \left\{p_{
\theta}(x)=\frac{b}{a}\left(\frac{x}{a}\right)^{b-1}e^{-\left(\frac{x}{a}\right)^{b}},\left.
                        \begin{array}{ll}
                         \theta= (a , b )\in \mathrm{I\! R}_{+}\times \mathrm{I\! R}_{+} & \hbox{} \\
                          x\in \mathrm{I\! R}_{+} & \hbox{}
                        \end{array}
                      \right.
\right\}$ be a  Weibull statistical  model where $p_{ \theta}$ is
Weibull density function. The coordinate system on  Weibull manifold
does not admit dual coordinates or potential function.

\end{proposition}
\begin{proof}
We determine the dual coordinates with respect to biorthogonality
condition
\[
 g\left(\partial_{\theta_{1}}l(x,\theta),\partial_{\eta_{1}}l(x,\theta)\right)
=-\mathrm{I\!
E}\left[\partial_{\eta_{1}}\left(-a^{-1}b+a^{-1-b}bx^{b}\right)\right]=1.\]
\[
g\left(\partial_{\theta_{1}}l(x,\theta),\partial_{\eta_{2}}l(x,\theta)\right)=\mathrm{I\!
E}\left[\partial_{\eta_{2}}\left(-a^{-1}b+a^{-1-b}bx^{b}\right)\right]=0.\]
\[g\left(\partial_{\theta_{2}}l(x,\theta),\partial_{\eta_{1}}l(x,\theta)\right)=-
\mathrm{I\!
E}\left[\partial_{\eta_{1}}\left(a^{-b}x^{b}\log(a)-a^{-b}x^{b}\log(x)+b^{-1}+\log(x)-\log(a)\right)\right]=0.
  \]
  \[g\left(\partial_{\theta_{2}}l(x,\theta),\partial_{\eta_{2}}l(x,\theta)\right)=-
\mathrm{I\!
E}\left[\partial_{\eta_{2}}\left(a^{-b}x^{b}\log(a)-a^{-b}x^{b}\log(x)+b^{-1}+\log(x)-\log(a)\right)\right]=1.
  \]
we obtain the following system $\left\{
   \begin{array}{ll}
     \partial_{\eta_{1}}\left(0\right)=-1 & \hbox{} \\
      \partial_{\eta_{1}}\left(0\right)=0 & \hbox{} \\
      \partial_{\eta_{2}}\left(0\right)=0 & \hbox{} \\
      \partial_{\eta_{2}}\left(0\right)=-1.& \hbox{}
   \end{array}
 \right.
$\\ What is impossible to solve. So, the following system
\begin{equation*}
    \left\{
       \begin{array}{ll}
         \frac{\partial^{2}\Phi}{\partial \alpha^{2}}=- \frac{\beta^{2}}{ \alpha^{2}}& \hbox{} \\
       \frac{\partial^{2}\Phi}{\partial \alpha\partial \beta}=- \frac{\varrho_{1}-1}{ \alpha^{2}} & \hbox{} \\
         \frac{\partial^{2}\Phi}{\partial \alpha^{2}}=- \frac{\beta \pi^{2}-6\alpha^{2}\varrho_{2}}{ 6\alpha^{2}} & \hbox{}
       \end{array}
     \right.
\end{equation*}
has no solution, where $\Phi$ is the  potential function sought.

\end{proof}

\section{Potential function and gradient system on Weibull logit manifold.}\label{sec5}

\begin{definition}
On $\mathrm{I\! R}$, and  for all $X=(m,n)$,
$Y=(m',n')\in\mathrm{I\! R}^{2}$ we define the $\iota$-product by
\begin{eqnarray*}
 \iota:\mathrm{I\! R}^{2}\times\mathrm{I\! R}^{2}&\longrightarrow&\mathrm{I\! R}^{2} \\
  (X,Y)&=&\iota\left(X,Y\right)=\left(m'^{\frac{1}{n}}m, n n'\right)
\end{eqnarray*}
\end{definition}
So, we have the following proposition
\begin{proposition}
On $\mathrm{I\! R}^{2}$, $\nu$ given by \begin{eqnarray*}
\nu:\mathrm{I\! R}^{2}\times\mathrm{I\! R}&\longrightarrow&\mathrm{I\! R} \\
  (\theta,x)&=&x\cdot\theta =a^{-b}x^{b}
\end{eqnarray*}
is the action.
\end{proposition}

\begin{proof}
Soit $e=(1,1)$ the neuter element in $(\mathrm{I\! R}^{2},\iota)$.
We have
\begin{equation*}
    \nu\left(e\right)(x)=x
\end{equation*}
and let $Y=(a,b), X=(a',b')\in \mathrm{I\! R}^{2}$ we have
\begin{eqnarray*}
   \nu\left(X,\nu\left(Y\right)(x)\right) &=&\nu\left(X,a^{-b}x^{b}
   \right) \\
   &=&a'^{-b'}\left(a^{-b}x^{b}\right)^{b'} \\
  &=& a'^{-b'}a^{-bb'}x^{bb'}
\end{eqnarray*}
the same we have
\begin{eqnarray*}
   \nu\left(X,Y\right)(x) &=&\left(a'^{\frac{1}{b}}a\right)^{bb'}x^{bb'}  \\
   &=&a'^{-b'}a^{-bb'}x^{bb'}
\end{eqnarray*}
So, \begin{eqnarray*}
   \nu\left(X,\nu\left(Y\right)(x)\right) &=&
   \nu\left(X,Y\right)(x)\end{eqnarray*}.
\end{proof}

We have the following theorem

\begin{theorem}
For all  Weibull statistical  manifold\\ $S = \left\{p_{
\theta}(x)=\frac{b}{a}\left(\frac{x}{a}\right)^{b-1}e^{-\left(\frac{x}{a}\right)^{b}},\left.
                        \begin{array}{ll}
                         \theta= (a , b )\in \mathrm{I\! R}_{+}\times \mathrm{I\! R}_{+} & \hbox{} \\
                          x\in \mathrm{I\! R}_{+} & \hbox{}
                        \end{array}
                      \right.
\right\}$   where $p_{ \theta}$ is Weibull density function, there
exist the logit model\\ $S' = \left\{p_{
\theta}(y,x)=\frac{b}{2a}\left(\frac{x}{a}\right)^{b-1}e^{-\left(\frac{x}{a}\right)^{b}},\left.
                        \begin{array}{ll}
                         \theta= (a , b )\in \mathrm{I\! R}_{+}\times \mathrm{I\! R}_{+} & \hbox{} \\
                          x\in \mathrm{I\! R}_{+} & \hbox{}
                        \end{array}
                      \right.
\right\}$ with the fundamental condition on the variable
\[\displaystyle {\it x}={\it RootOf} \left( 2\,b\,{a}^{-b}{{\it \_Z}}^{b}-2\,b-2\,{a}^{-b}{{\it \_Z}}^{b}a\,\ln  \left( a \right) \\
\mbox{}+2\,{a}^{-b}{{\it \_Z}}^{b}a\,\ln  \left( {\it \_Z} \right)
-2\,\ln  \left( {\it \_Z} \right) a+2\,\ln \left( a \right) a-a
\right)\] admitting the potential function \begin{eqnarray*}
  \Phi(\theta) &=&
  \frac{b^{2}}{12a^{2}x}\left(a^{-b}x^{b}-1\right)^{4}+\frac{b^{2}}{3a^{2}x}a^{-b}x^{b}-\frac{b^{2}}{12a^{2}x}.
\end{eqnarray*}, the dual coordinate system given by
    $\left(\eta_{1},\eta_{2}\right)=\left(\xi_{1}(\theta),\xi_{2}(\theta)\right),$
and dual potential function
\begin{equation*}
  \Psi(\eta)= a\xi_{1}+b\xi_{2}-  \frac{b^{2}}{12a^{2}x}\left(a^{-b}x^{b}-1\right)^{4}+\frac{b^{2}}{3a^{2}x}a^{-b}x^{b}-\frac{b^{2}}{12a^{2}x}
\end{equation*}
 satisfy the Legendre equation
\begin{equation*}
   \theta_{1}\eta_{1}+\theta_{2}\eta_{2}- \Phi(\theta)-\Psi(\eta)=0
\end{equation*}Where $\xi_{1}(\theta)=\frac{\partial \Phi}{\partial a},\;\xi_{2}(\theta)=\frac{\partial \Phi}{\partial b}$
\end{theorem}
\begin{proof}

According to Hisatoshi Tanaka \cite{hisatoshi}, we define the new
variable
\begin{eqnarray*}
  y &=& \left\{
          \begin{array}{ll}
            1 & \hbox{if $y^{*}\geq 0$} \\
            0 & \hbox{$y^{*}<0$.}
          \end{array}
        \right.
\end{eqnarray*} where $y^{*}=x\cdot\theta-\epsilon$, such that $\mathrm{I\!
E}(\epsilon)=0$. The choice of conditional probability is given
by\begin{eqnarray*}
  F(x\cdot\theta) &=& P\left\{y=1/x\right\}=P\left\{\epsilon\leq
  a^{-b}x^{b}/x\right\}=\frac{1}{2}.
\end{eqnarray*}
We define the binary probability density
\begin{equation*}
    p_{\theta}(y,x)=\frac{1}{2}p_{\theta}(x)
\end{equation*} we have
\begin{equation*}
    p_{\theta}(y,x)=\frac{b}{2a}\left(\frac{x}{a}\right)^{b-1}e^{-\left(\frac{x}{a}\right)^{b}}
\end{equation*}
where $(y,x)\in \{0,1\}\times \mathrm{I\! R}$. We have
\begin{eqnarray*}
  \log  p_{\theta}(y,x)&=& -\log 2+\log b-\log a-(b-1)\log
  a+(b-1)\log x-x\cdot\theta.
\end{eqnarray*}
We obtain the following relation
\begin{equation*}
    \left\{
      \begin{array}{ll}
        \frac{\partial  \log  p_{\theta}(y,x)}{\partial  a}= \frac{b}{a}\left(x\cdot\theta-1\right)& \hbox{} \\
         \frac{\partial  \log  p_{\theta}(y,x)}{\partial  b}= x\cdot\theta\left(\log a-\log x\right)+\log x-\log a+\frac{1}{b}& \hbox{.}
      \end{array}
    \right.
\end{equation*}
In this we find the following  function $f(x\cdot \theta)$ satisfy
the following system
\begin{equation*}
    \left\{
      \begin{array}{ll}
        \frac{\partial  \log  p_{\theta}(y,x)}{\partial  a}= \frac{y-F(x\cdot\theta)}{F(x\cdot\theta)\left(1-F(x\cdot\theta)\right)}f(x\cdot\theta)x& \hbox{} \\
         \frac{\partial  \log  p_{\theta}(y,x)}{\partial  b}= \frac{y-F(x\cdot\theta)}{F(x\cdot\theta)\left(1-F(x\cdot\theta)\right)}f(x\cdot\theta)x& \hbox{.}
      \end{array}
    \right.
\end{equation*}
we have

\begin{equation*}
    \left\{
      \begin{array}{ll}
        \frac{\partial  \log  p_{\theta}(y,x)}{\partial  a}= 2f(x\cdot\theta)x& \hbox{} \\
         \frac{\partial  \log  p_{\theta}(y,x)}{\partial  b}= 2f(x\cdot\theta)x& \hbox{.}
      \end{array}
    \right.
\end{equation*}
we obtain $f(x\cdot\theta)=\frac{\frac{b}{a}(x\cdot\theta-1}{2x}$
with the following condition on variable $x$ that is
\[\displaystyle {\it x}={\it RootOf} \left( 2\,b\,{a}^{-b}{{\it \_Z}}^{b}-2\,b-2\,{a}^{-b}{{\it \_Z}}^{b}a\,\ln  \left( a \right) \\
\mbox{}+2\,{a}^{-b}{{\it \_Z}}^{b}a\,\ln  \left( {\it \_Z} \right)
-2\,\ln  \left( {\it \_Z} \right) a+2\,\ln \left( a \right) a-a
\right).\] So by setting $r(u)=\frac{\frac{b}{a}(u-1}{2x}$, and with
$u=x\cdot\theta$.\\According to Hisatoshi Tanaka \cite{hisatoshi},
we have the potential function given by
\begin{eqnarray}\label{44}
  \Phi(\theta) &=& \mathrm{I\!
E}\left[\int_{0}^{x\cdot\theta}\left(\int_{0}^{v}r(u)du\right)dv\right]
\end{eqnarray}
the equation (\ref{44}) becomes
\begin{eqnarray}\label{44}
  \Phi(\theta) &=& \mathrm{I\!
E}\left[\int_{0}^{x\cdot\theta}\left(\int_{0}^{v}\frac{\frac{b}{a}(u-1)}{2x}du\right)dv\right].
\end{eqnarray}
We obtain
\begin{eqnarray*}
  \Phi(\theta) &=&
  \frac{b^{2}}{12a^{2}x}\left(a^{-b}x^{b}-1\right)^{4}+\frac{b^{2}}{3a^{2}x}a^{-b}x^{b}-\frac{b^{2}}{12a^{2}x}.
\end{eqnarray*}
\end{proof}

We have the following proposition
\begin{proposition}\label{pr1}
Let $S' = \left\{p_{
\theta}(y,x)=\frac{b}{2a}\left(\frac{x}{a}\right)^{b-1}e^{-\left(\frac{x}{a}\right)^{b}},\left.
                        \begin{array}{ll}
                         \theta= (a , b )\in \mathrm{I\! R}_{+}\times \mathrm{I\! R}_{+} & \hbox{} \\
                          x\in \mathrm{I\! R}_{+} & \hbox{}
                        \end{array}
                      \right.
\right\}$ logit Weibull manifold. The information metric on logit
Weibull manifold  is given by
\begin{equation*}
    I(\theta)=\left(
                \begin{array}{cc}
                  -\frac{\partial^{2}\Phi(\theta)}{\partial a^{2}} & -\frac{\partial^{2}\Phi(\theta)}{\partial a\partial b} \\
                  -\frac{\partial^{2}\Phi(\theta)}{\partial a\partial b} & -\frac{\partial^{2}\Phi(\theta)}{\partial b^{2}} \\
                \end{array}
              \right)
\end{equation*}, and the inverse of information geometric is given
by \begin{equation*}
    I^{-1}(\theta)=\left(
                \begin{array}{cc}
                  -\frac{1}{\mathcal{A}}\frac{\partial^{2}\Phi(\theta)}{\partial b^{2}} & \frac{1}{\mathcal{A}}\frac{\partial^{2}\Phi(\theta)}{\partial a\partial b} \\
                  \frac{1}{\mathcal{A}}\frac{\partial^{2}\Phi(\theta)}{\partial a\partial b} & -\frac{1}{\mathcal{A}}\frac{\partial^{2}\Phi(\theta)}{\partial a^{2}} \\
                \end{array}
              \right)
\end{equation*}
with $\mathcal{A}=\frac{\partial^{2}\Phi(\theta)}{\partial
a^{2}}.\frac{\partial^{2}\Phi(\theta)}{\partial
b^{2}}-\left(\frac{\partial^{2}\Phi(\theta)}{\partial a\partial
b}\right)^{2}$, and  \begin{eqnarray*}
  \Phi(\theta) &=&
  \frac{b^{2}}{12a^{2}x}\left(a^{-b}x^{b}-1\right)^{4}+\frac{b^{2}}{3a^{2}x}a^{-b}x^{b}-\frac{b^{2}}{12a^{2}x}.
\end{eqnarray*} and \[\displaystyle {\it x}={\it RootOf} \left( 2\,b\,{a}^{-b}{{\it \_Z}}^{b}-2\,b-2\,{a}^{-b}{{\it \_Z}}^{b}a\,\ln  \left( a \right) \\
\mbox{}+2\,{a}^{-b}{{\it \_Z}}^{b}a\,\ln  \left( {\it \_Z} \right)
-2\,\ln  \left( {\it \_Z} \right) a+2\,\ln \left( a \right) a-a
\right)\]
\end{proposition}
\begin{proof}
Apply the Amari theorem \cite{Shu}, we have the result.
\end{proof}

The following proposition leads us to the following result
\begin{proposition}\label{pr2}
Let $S' = \left\{p_{
\theta}(y,x)=\frac{b}{2a}\left(\frac{x}{a}\right)^{b-1}e^{-\left(\frac{x}{a}\right)^{b}},\left.
                        \begin{array}{ll}
                         \theta= (a , b )\in \mathrm{I\! R}_{+}\times \mathrm{I\! R}_{+} & \hbox{} \\
                          x\in \mathrm{I\! R}_{+} & \hbox{}
                        \end{array}
                      \right.
\right\}$ the logit model on Weibull manifold. The gradient system
on logit Weibull manifold is given by
\begin{equation*}
    \left\{
  \begin{array}{ll}
     \dot{a}=\frac{1}{\mathcal{A}}\frac{\partial^{2}\Phi(\theta)}{\partial b^{2}} .\frac{\partial\Phi(\theta)}{\partial a}
     - \frac{1}{\mathcal{A}}\frac{\partial^{2}\Phi(\theta)}{\partial a\partial b}.\frac{\partial\Phi(\theta)}{\partial b}& \hbox{} \\
    \dot{b}= -\frac{1}{\mathcal{A}}\frac{\partial^{2}\Phi(\theta)}{\partial a\partial b}.\frac{\partial\Phi(\theta)}{\partial a}
     +\frac{1}{\mathcal{A}}\frac{\partial^{2}\Phi(\theta)}{\partial a^{2}}.\frac{\partial\Phi(\theta)}{\partial b}& \hbox{}
  \end{array}
\right.
\end{equation*}
where \begin{eqnarray*}
  \Phi(\theta) &=&
  \frac{b^{2}}{12a^{2}x}\left(a^{-b}x^{b}-1\right)^{4}+\frac{b^{2}}{3a^{2}x}a^{-b}x^{b}-\frac{b^{2}}{12a^{2}x}.
\end{eqnarray*} and \[\displaystyle {\it x}={\it RootOf} \left( 2\,b\,{a}^{-b}{{\it \_Z}}^{b}-2\,b-2\,{a}^{-b}{{\it \_Z}}^{b}a\,\ln  \left( a \right) \\
\mbox{}+2\,{a}^{-b}{{\it \_Z}}^{b}a\,\ln  \left( {\it \_Z} \right)
-2\,\ln  \left( {\it \_Z} \right) a+2\,\ln \left( a \right) a-a
\right)\]
\end{proposition}

\begin{proof}
Using (\ref{e2}), and the proposition \ref{pr1} we have the result.
\end{proof}

\vfill\eject

\section{General conclusion}\label{sec8}
In this paper we asked whether there exists a gradient system
defined on the variety constructed from Weibull distributions. We
have shown that there is no function on this variety to construct a
gradient system. But that there is a hybrid Weibull model based on
the choice of the Weibull probability. On the variety defined from
this new Weibull density, which we have called the logit density, we
have shown that there is a gradient system on this variety. Since we
are in dimension and by applying the Fujiwara \cite{Dy} and Nakamura
\cite{nakamura} results we can show that it is a Hamiltonian system
and completely integrable by apply the Liouville theorem \cite{Li}.

\backmatter

\bmhead{Supplementary information} This manuscript has no additional
data.

\bmhead{Acknowledgments} I would like to thank all the members of
the Algebra and Geometry Laboratory of the University of
Yaound$\acute{e}$ 1.
\section*{Declarations}
This article has no conflict of interest to the journal. No
financing with a third party.
\begin{itemize}
\item No Funding
\item No Conflict of interest/Competing interests (check journal-specific guidelines for which heading to use)
\item  Ethics approval
\item  Consent to participate
\item  Consent for publication
\item  Availability of data and materials
\item  Code availability
\item Authors' contributions
\end{itemize}

\noindent



\bigskip\noindent

\bibliography{sn-bibliography}

\end{document}